\documentclass[12pt]{amsart}
\usepackage{amssymb}
\usepackage{verbatim}
\usepackage{hyperref}
\usepackage{color}

\begin{document}

\newtheorem{theorem}{Theorem}    
\newtheorem{proposition}[theorem]{Proposition}
\newtheorem{conjecture}[theorem]{Conjecture}
\def\theconjecture{\unskip}
\newtheorem{corollary}[theorem]{Corollary}
\newtheorem{lemma}[theorem]{Lemma}
\newtheorem{sublemma}[theorem]{Sublemma}
\newtheorem{observation}[theorem]{Observation}
\theoremstyle{definition}
\newtheorem{definition}{Definition}
\newtheorem{notation}[definition]{Notation}
\newtheorem{remark}[definition]{Remark}
\newtheorem{question}[definition]{Question}
\newtheorem{questions}[definition]{Questions}
\newtheorem{example}[definition]{Example}
\newtheorem{problem}[definition]{Problem}
\newtheorem{exercise}[definition]{Exercise}

\numberwithin{theorem}{section}
\numberwithin{definition}{section}
\numberwithin{equation}{section}

\def\earrow{{\mathbf e}}
\def\rarrow{{\mathbf r}}
\def\uarrow{{\mathbf u}}
\def\varrow{{\mathbf V}}
\def\tpar{T_{\rm par}}
\def\apar{A_{\rm par}}

\def\reals{{\mathbb R}}
\def\torus{{\mathbb T}}
\def\heis{{\mathbb H}}
\def\integers{{\mathbb Z}}
\def\naturals{{\mathbb N}}
\def\complex{{\mathbb C}\/}
\def\distance{\operatorname{distance}\,}
\def\support{\operatorname{support}\,}
\def\dist{\operatorname{dist}\,}
\def\Span{\operatorname{span}\,}
\def\degree{\operatorname{degree}\,}
\def\kernel{\operatorname{kernel}\,}
\def\dim{\operatorname{dim}\,}
\def\codim{\operatorname{codim}}
\def\trace{\operatorname{trace\,}}
\def\Span{\operatorname{span}\,}
\def\dimension{\operatorname{dimension}\,}
\def\codimension{\operatorname{codimension}\,}
\def\nullspace{\scriptk}
\def\kernel{\operatorname{Ker}}
\def\ZZ{ {\mathbb Z} }
\def\p{\partial}
\def\rp{{ ^{-1} }}
\def\Re{\operatorname{Re\,} }
\def\Im{\operatorname{Im\,} }
\def\ov{\overline}
\def\eps{\varepsilon}
\def\lt{L^2}
\def\diver{\operatorname{div}}
\def\curl{\operatorname{curl}}
\def\etta{\eta}
\newcommand{\norm}[1]{ \|  #1 \|}
\def\expect{\mathbb E}
\def\bull{$\bullet$\ }

\def\xone{x_1}
\def\xtwo{x_2}
\def\xq{x_2+x_1^2}
\newcommand{\abr}[1]{ \langle  #1 \rangle}

\newcommand{\Norm}[1]{ \left\|  #1 \right\| }
\newcommand{\set}[1]{ \left\{ #1 \right\} }
\def\one{\mathbf 1}
\def\whole{\mathbf V}
\newcommand{\modulo}[2]{[#1]_{#2}}

\def\scriptf{{\mathcal F}}
\def\scriptg{{\mathcal G}}
\def\scriptm{{\mathcal M}}
\def\scriptb{{\mathcal B}}
\def\scriptc{{\mathcal C}}
\def\scriptt{{\mathcal T}}
\def\scripti{{\mathcal I}}
\def\scripte{{\mathcal E}}
\def\scriptv{{\mathcal V}}
\def\scriptw{{\mathcal W}}
\def\scriptu{{\mathcal U}}
\def\scriptS{{\mathcal S}}
\def\scripta{{\mathcal A}}
\def\scriptr{{\mathcal R}}
\def\scripto{{\mathcal O}}
\def\scripth{{\mathcal H}}
\def\scriptd{{\mathcal D}}
\def\scriptl{{\mathcal L}}
\def\scriptn{{\mathcal N}}
\def\scriptp{{\mathcal P}}
\def\scriptk{{\mathcal K}}
\def\frakv{{\mathfrak V}}

\author{Michael Christ}
\address{
        Michael Christ\\
        Department of Mathematics\\
        University of California \\
        Berkeley, CA 94720-3840, USA}
\email{mchrist@math.berkeley.edu}
\thanks{Research supported in part by NSF grant DMS-0901569.
The second author was supported partly by NSFC (Grant No.10701010), NSFC
(Key program Grant No.10931001), Beijing Natural Science Foundation
(Grant: 1102023), Program for Changjiang Scholars and Innovative Research
Team in University.}

\keywords{Extremizers, Euler-Lagrange equation, smoothness, multilinear bounds,
weighted norm inequalities, $A_p$ weights.  }


\author{Qingying Xue}
\address{Qingying Xue
\\
School of Mathematical Sciences
\\
Beijing Normal University
\\
Laboratory of Mathematics and Complex Systems
\\
Ministry of Education
\\
Beijing 100875
\\
People's Republic of China
}
\email{qyxue@bnu.edu.cn}

\date{December 24, 2010.}

\title
[Smoothness of Extremizers]
{Smoothness of Extremizers \\ of a Convolution Inequality}

\maketitle

\begin{abstract}
Let $d\ge 2$ and $T$ be the convolution operator
$Tf(x)=\int_{\reals^{d-1}} f(x'-t,x_d-|t|^2)\,dt$,
which is is bounded from $L^{(d+1)/d}(\reals^d)$
to $L^{d+1}(\reals^d)$.
We show that any critical point $f\in L^{(d+1)/d}$
of the functional
$\norm{Tf}_{d+1}/\norm{f}_{(d+1)/d}$
is infinitely differentiable,
and that $|x|^\delta f\in L^{(d+1)/d}$
for some $\delta>0$.
In particular, this holds for all extremizers of the associated
inequality.
This is done by exploiting a generalized Euler-Lagrange equation,
and certain weighted norm inequalities for $T$.
\end{abstract}

\section{Introduction}

Optimal constants and extremizers have been determined
for some of the most fundamental $L^p$ inequalities
of Fourier and real analysis. Among such achievements are the
celebrated works of Beckner \cite{beckner}, Burkholder \cite{burkholder},
Lieb \cite{liebfractional},\cite{liebgaussian}, and Pichorides \cite{pichorides}.
Certain multilinear inequalities, governed by linear geometric structure,
have more recently been treated in \cite{BCCT}.
Still more recently,
optimal constants and extremizers have been determined
for Fourier restriction/extension inequalities
for paraboloids, in the lowest dimensions,
in works of Foschi \cite{foschi}, Hundertmark and Zharnitsky
\cite{HZ}, and  Bennett, Bez, Carbery, and Hundertmark \cite{hotstrichartz}.
The geometry which underlies restriction inequalities features curvature.

The present paper is one of a series
\cite{christshao1},\cite{christshao2},\cite{christquilodran},\cite{comegaextremals},\cite{christextremal},\cite{quasiextremal},\cite{betsyquasiextremal}
which treat questions concerning extremals for
certain $L^p$ norm inequalities,
whose form is determined by the influence of curvature and singularities.
These works focus on less fine questions such as the existence of extremizers,
precompactness of extremizing sequences,
and qualitative and quantitative properties of extremizers.
The present paper is concerned with such properties of
extremizers, for one particular inequality.

Let $d\ge 2$.
Points of $\reals^d$ will be represented
as $x=(x',x_d)\in\reals^{d-1}\times\reals^1$.
Our object of investigation is the convolution operator
\[
Tf(x)=\int_{\reals^{d-1}} f(x'-t,x_d-|t|^2)\,dt.
\]
This operator is bounded from $L^{(d+1)/d}(\reals^d)$
to $L^{d+1}(\reals^d)$, and
satisfies no other $L^p\to L^q$ inequalities.
The curvature of the parabola $x_d=|x'|^2$ and scaling symmetry
of the measure $dx'|_{x_d=|x'|^2}$
are the crucial ingredients in this theory.
This operator $T$ enjoys a rich symmetry structure discussed in
\cite{quasiextremal},
and is perhaps the most prototypical representative of the class of
operators $f\mapsto f*\mu$, where $\mu$ is a measure supported on
a nonflat submanifold of $\reals^d$.

Let $p_0=p_0(d)=(d+1)/d$ and $q_0=q_0(d)=d+1$.
Denote by ${\mathbf A}_d$
the optimal constant in the inequality
\begin{equation} \label{theinequality}
\norm{Tf}_{q_0}\le {\mathbf A}_d\norm{f}_{p_0}.
\end{equation}
An $\eps$-quasiextremal for inequality \eqref{theinequality}
is a function satisfying $\norm{Tf}_{q_0}\ge \eps\norm{f}_{p_0}$.
A characterization of quasiextremals is established in \cite{quasiextremal},
which includes some quantitative though non-optimal control as $\eps\to 0$.
It is shown in \cite{christextremal} that extremizers for the inequality
\eqref{theinequality} exist, and that any nonnegative extremizing
sequence of functions is precompact modulo action of the group of all
geometric symmetries of the inequality.

In the present paper we take a third step by establishing two properties
of extremizers: smoothness, and some improved decay.
These are established for all critical points of the functional
$\norm{Tf}_{q_0}/\norm{f}_{p_0}$.
We formulate a conjecture concerning the precise decay rate of
nonnegative extremizers.
The extremizers and optimal constant ${\mathbf A}_d$ remain unknown,
and it remains unknown whether extremizers are unique modulo natural symmetries.

A technical device which underlies the analysis, and which may be of some independent interest,
is a family of weighted norm inequalities for $T$.
Lemma~\ref{lemma:weighted} formulates
a one parameter family of
rather sharp weighted inequalities. These involve pairs of exponents
different from $(p_0,q_0)$, are not consequences of \eqref{theinequality},
and are suited to our purpose.

The transpose $T^*$ of $T$  takes the form
\[
T^*f(x)
=\int_{\reals^{d-1}} f(x'+t,x_d+|t|^2)\,dt
=\int_{\reals^{d-1}} f(x'-t,x_d+|t|^2)\,dt.
\]
$T^*$ is equal to $T$  conjugated with the
norm-preserving operator associated to the transformation
$(y',y_d)\mapsto (y',-y_d)$ of $\reals^d$. In particular, $T^*$
is likewise bounded
from $L^{(d+1)/d}(\reals^d)$ to $L^{d+1}(\reals^d)$,

Real-valued critical points of the functional
$\norm{Tf}_{q_0}/\norm{f}_{p_0}$
are characterized by the generalized Euler-Lagrange equation
\begin{equation} \label{eq:EL}
f = \lambda \Big(T^*\big[(Tf)^d \big] \Big)^d
\end{equation}
where
\begin{equation}
\lambda = A_d^{-dq_0} \norm{f}_{p_0}^{dp_0-dq_0}.
\end{equation}
Complex-valued critical points are characterized by this same equation
\eqref{eq:EL} with $\lambda  = \norm{Tf}_{q_0}^{-dq_0}\norm{f}_{p_0}^{dp_0}$,
provided that powers of complex numbers on the right-hand side of \eqref{eq:EL}
are interpreted as follows:
If $z\in\complex$ and $0\ne s\in\reals$,
then $z^s$ should be interpreted as $z|z|^{s-1}$.
When $s=d$ is an even integer,
this is not a product of positive integer powers of $z$ and $\bar z$.

The main result of this paper is:
\begin{theorem} \label{thm:main}
Let $d\ge 2$, and let $\lambda\in\reals$.
Let $f\in L^{p_0(d)}(\reals^d)$ be any
real-valued solution of the generalized Euler-Lagrange equation \eqref{eq:EL}.
Then $f\in C^\infty$.
Moreover, all partial derivatives of $f$ are bounded functions,
and there exists $\delta>0$ such that
$(1+|x|)^\delta \nabla^k f(x)\in L^{p_0}(\reals^d)$ for all $k\ge 0$.

The same conclusion holds for all complex-valued solutions if $d$ is even
and $\lambda\in\complex$.
\end{theorem}
This means, of course, that there exists a $C^\infty$ function which is equal
almost everywhere to $f$.

Inequality \eqref{theinequality} is invariant under parabolic scaling.
There are no {\it a priori} inequalities which assert that
if $f\in L^{p_0}$, then $S(f)=(T^*[(Tf)^d])^d$ has additional decay
or smoothness properties, which would lead to a simple proof of
the theorem via a boostrapping argument.
Instead, we will analyze the linearization (and all but the highest order
terms in its finite Taylor series) of the multilinear operator $S$
about a dense class of functions, and will show that these operators
do improve decay. The key in using this fact in conjunction with \eqref{eq:EL}
and a fixed-point argument, is to find Banach spaces which encode
more rapid decay than does $L^{(d+1)/d}$, and which are preserved by $S$.
We do this by developing a limited theory of weighted inequalities for $T$.

We will demonstrate Theorem~\ref{thm:main} only in the real case. The same
reasoning applies to the complex case, with small and straightforward modifications
in formulas to accommodate various complex conjugations.

\medskip
Define
\begin{align}
\upsilon(x) &= \min\big(1,|x'|^{-d},|x_d-|x'|^2|^{-d} \big)
\\
\upsilon_*(x) &= \min\big(1,|x'|^{-d},|x_d+|x'|^2|^{-d} \big)
\end{align}

The functions $\upsilon,\upsilon_*$ are $O(|x|^{-\delta})$
as $|x|\to\infty$ for some $\delta>0$,
and have another noteworthy aspect.
Consider the parabolic dilations $\delta_r(x',x_d) = (rx',r^2x_d)$, for $r>0$,
and the associated operators $\delta_r(f)(x) = f(\delta_r(x))$.
With respect to these dilations, $T$ enjoys the symmetry
$\delta_r(Tf) \equiv r^{d-1}T(\delta_r(f))$.
The weight $v$ equals $w^{-d}$ where $w(x)$ is the maximum of
the three quantities $w_0(x)=1$, $w_1(x)=|x'|$ and $w_2(x)=|x_d-|x'|^2|$.
Each is homogeneous with respect to the dilations $\delta_r$,
but $w_j$ is homogeneous of degree $j$ for $j=0,1,2$.\footnote{
The three are related: $|w_j(x)-w_j(\tilde x)|=O(w_{j-1}(x))$
if $|x-\tilde x|=O(1)$, for $j=1,2$.}
We believe that $\upsilon$ accurately expresses the
behavior of extremals, in the following sense.

\begin{conjecture}
Let $\lambda\in\complex$, and let $f\in L^{p_0(d)}(\reals^d)$
be any solution of the generalized Euler-Lagrange equation \eqref{eq:EL}.

(i)
There exists $C=C(f,\lambda)<\infty$
such that  for almost every $x\in\reals^d$,
$f(x)\le C\upsilon(x)$.

(ii)
If $\lambda>0$, $f$ is nonnegative and $\norm{f}_{p_0(d)}>0$,
then there exists $c=c(f,\lambda)>0$ such that for almost every $x\in\reals^d$,
$f(x)\ge c\upsilon(x)$.
\end{conjecture}

\section{Weighted inequalities}

The following elementary inequalities provide a foundation for our analysis.
The proof of Lemma~\ref{lemma:key},
deferred to \S\ref{section:tedious}, is thoroughly elementary but is not short.
\begin{lemma} \label{lemma:key}
For each $d\ge 2$ there exists $C_d<\infty$ such that
\begin{align}
T(\upsilon_*)&\le C_d \upsilon^{1/d}
\\
T^*(\upsilon)&\le C_d \upsilon_*^{1/d}.
\end{align}
\end{lemma}

For $\theta\in[0,1]$ define exponents $p_\theta,q_\theta$
to be
\begin{align}
p_\theta^{-1} &= (1-\theta)p_0(d)^{-1}
\\
q_\theta^{-1} &= (1-\theta)q_0(d)^{-1}.
\end{align}
Then $q_\theta = p_\theta/d$.

\begin{lemma}[Weighted Inequalities] \label{lemma:weighted}
There exists $C<\infty$
such that for every $t\in[0,1]$,
for every nonnegative function $f$,
\begin{align}
\Big(\int_{\reals^d} (Tf)^{q_t} \upsilon^{-tq_t/d}\Big)^{1/q_t}
&\le C
\Big(\int_{\reals^d} f^{p_t} \upsilon_*^{-tp_t}\Big)^{1/p_t}.
\\
\Big(\int_{\reals^d} (T^*f)^{q_t} \upsilon_*^{-tq_t/d}\Big)^{1/q_t}
&\le C
\Big(\int_{\reals^d} f^{p_t} \upsilon^{-tp_t}\Big)^{1/p_t}.
\end{align}
\end{lemma}

\begin{proof}
Consider the analytic family of operators
$T_z f = \upsilon^{-z/d}T(\upsilon_*^{z}f)$
on the strip $\{z: 0\le\Re(z)\le 1\}$.
When $\Re(z)=0$, $T_z$ is bounded from
$L^{p_0}$ to $L^{q_0}$.
When $\Re(z)=1$, $T_z$ is bounded from
$L^\infty$ to $L^\infty$, by Lemma~\ref{lemma:key}.
Both bounds are uniform in $\Im(z)$.
The first conclusion of Lemma~\ref{lemma:weighted} follows
by complex interpolation.
The inequality for $T^*$ is proved in the same way.
\end{proof}

Lemma~\ref{lemma:weighted} has consequences which are conveniently expressed
in terms of certain nonlinear operators and certain function spaces.
Define nonlinear operators
\begin{align}
\scriptt(f) &= (Tf)^d,
\\
\scriptt_*(f) &= (T^*f)^d
\\
S(f) &= \scriptt_*(\scriptt(f)).
\end{align}
In these terms, the Euler-Lagrange equation \eqref{eq:EL} becomes $f=\lambda S(f)$.

Set
\begin{equation}
w=\upsilon_*^{-1} \text{ and } w_*=\upsilon^{-1}.
\end{equation}
The following scales of Banach spaces
$X_\theta, X_{*,\theta}, Y_{*,\theta}$ are adapted to $T$ and $T^*$.
Define these spaces
to be the sets of all equivalence class of measurable functions
on $\reals^d$ for which the following weighted norms are finite:
\begin{align}
\norm{f}_{X_\theta}^{p_\theta} &=\int_{\reals^d} |f|^{p_\theta} w^{\theta p_\theta}
\\
\norm{f}_{X_{*,\theta}}^{p_\theta} &=\int_{\reals^d} |f|^{p_\theta} w_*^{\theta p_\theta}.
\\
\norm{f}_{Y_{*,\theta}}^{q_\theta} &=\int_{\reals^d} |f|^{q_\theta} w_*^{\theta q_\theta/d}.
\end{align}
In particular, $X_{0}=X_{*,0}=L^{p_0(d)}$ and $Y_{*,0}=L^{q_0(d)}$.

These spaces enjoy the following properties.
\begin{lemma} \label{lemma:Xproperties}
(i)
If $\alpha\le\beta$ then $X_\beta\subset X_\alpha$.
More precisely,
there exists $C<\infty$ such that for all $0\le\alpha\le\beta\le 1$,
for all $f\in X_\beta$,
\[\norm{f}_{X_\alpha}\le C\norm{f}_{X_\beta}.\]
(ii)
If $0\le\alpha<\beta\le 1$ then there exists $\delta>0$
such that
$f\in X_\beta\Rightarrow w^\delta f\in X_\alpha$.
\newline
(iii)
If $\alpha,\beta,\gamma,\theta\in[0,1]$ and $\gamma=\theta\alpha+(1-\theta)\beta$, then
\begin{equation}
\norm{f}_{X_{\gamma}}\le
\norm{f}_{X_{\alpha}}^{\theta}
\norm{f}_{X_{\beta}}^{1-\theta}
\end{equation}
for any $f\in X_\alpha\cap X_\beta$.
The function $\gamma\mapsto \norm{f}_{X_\gamma}$
is continuous on $[\alpha,\beta]$.
\newline
(iv)
Let $\alpha,\beta,\gamma,\theta\in[0,1]$ and $\gamma=\theta\alpha+(1-\theta)\beta$.
If a linear operator $L$ maps $X_\alpha\cap X_\beta$ to $X_\alpha+X_\beta$
and is bounded from $X_\alpha$ to $X_\alpha$, and from $X_\beta$ to $X_\beta$,
then
\begin{equation} \label{eq:logconvex}
\norm{L}_{X_\gamma\to X_\gamma}
\le
\norm{L}_{X_\alpha\to X_\alpha}^{\theta}
\norm{L}_{X_\beta\to X_\beta}^{1-\theta}.
\end{equation}
Here $\norm{L}_{X_t\to X_t}$ denotes the norm of $L$ as an operator from $X_t$ to $X_t$.
\end{lemma}

These conclusions are simple consequences of H\"older's inequality and complex interpolation.

\begin{corollary}
Let $\delta>0$ and suppose that $f\in X_\delta$.
Then the function $t\mapsto \norm{f}_{X_t}$ is continuous on $[0,\delta]$.
\end{corollary}
If $f\equiv 0$, this is trivial; otherwise
it is an immediate consequence of \eqref{eq:logconvex} and the fact
that any nonvanishing log-convex function is continuous.
\qed

Combining Lemmas~\ref{lemma:weighted} and \ref{lemma:Xproperties}
gives a result which will be useful in our proofs.
\begin{lemma}
For any $\theta\in[0,1]$,
$T$ maps $X_\theta$ to $Y_{*,\theta}$,
$\scriptt$ maps $X_\theta$ to $X_{*,\theta}$,
and
$\scriptt_*$ maps $X_{*,\theta}$ to $X_{\theta}$.
Therefore exists $C<\infty$ such that for all $\theta\in[0,1]$

\begin{align}
\norm{T(f)}_{Y_{*,\theta}}&\le C\norm{f}_{X_\theta}
\\
\norm{\scriptt(f)}_{X_{*,\theta}}&\le C\norm{f}_{X_\theta}^d
\\
\norm{\scriptt_*(f)}_{X_{\theta}}&\le C\norm{f}_{X_{*,\theta}}^d.
\end{align}
Likewise
$S$ maps $X_\theta$ to  $X_\theta$
and
\begin{equation}
\norm{S(f)}_{X_{\theta}} \le C\norm{f}_{X_\theta}^{d^2}
\end{equation}
for all $f\in X_\theta$.
\end{lemma}

We will need to
apply inequalities of Calder\'on-Zygmund/Littlewood-Paley type
at certain points in the proof, with respect to weighted $L^p$ norms.
There is a well-known condition on the weight which ensures that such
operators are bounded.
Denote by $A_p=A_p(\reals^d)$ the usual Muckenhoupt classes of weights \cite{stein}.
For $1<p<\infty$,
$A_p$ is the set of all locally integrable nonnegative functions $w$ for which the
quantity
\[
[u]_{A_p} = \sup_B \Big(|B|^{-1}\int_B u \Big)
\Big(|B|^{-1}\int_B u^{-1/(p-1)}\Big)^{p-1}
\]
is finite.
Operators of Calder\'on-Zygmund and Littlewood-Paley type
are bounded on $L^p(u)$ for $u\in A_p$ \cite{stein}.

\begin{lemma} \label{lemma:Ap}
Let $P>1$.
There exists $\delta>0$ such that
\[w^t\in A_p
\text{ for all $|t|\le\delta$ and $p\in[P,\infty]$.}\]
\end{lemma}

\begin{lemma} \label{lemma:maxofweights}
Let $u,v:\reals^d\to[0,\infty)$
be measurable functions and $p\in(1,\infty)$.
If $u,v\in A_p$ then $\max(u,v)\in A_p$.
\end{lemma}

\begin{proof}
Let $B\subset\reals^d$ be any ball of finite radius.
If $\int_B v\le \int_B u$, then the following reasoning applies:
\[
|B|^{-1}\textstyle\int_B
w
=
|B|^{-1}\textstyle\int_B
\max(u,v)
\le
|B|^{-1}\textstyle\int_B
(u+v)
\le 2
|B|^{-1}\textstyle\int_B
u.
\]
Therefore
\begin{align*}
|B|^{-1}\textstyle\int_B
w
\cdot
\Big(
|B|^{-1}\textstyle\int_B
w^{-1/(p-1)}\Big)^{p-1}
&\le
2|B|^{-1}\textstyle\int_B
u
\cdot
\Big(|B|^{-1}\textstyle\int_B
w^{-1/(p-1)}\Big)^{p-1}
\\
&\le
2|B|^{-1}\textstyle\int_B
u
\cdot
\Big(|B|^{-1}\textstyle\int_B
u^{-1/(p-1)}\Big)^{p-1}
\\
&\le 2[u]_{A_p}
\end{align*}
where $[u]_{A_p}$ is the $A_p$ constant of $u$.

If on the other hand
$\int_B v\ge \int_B u$, then the same reasoning yields the bound
$2[v]_{A_p}$.
Thus $[\max(u,v)]_{A_p}\le 2\max([u]_{A_p},[v]_{A_p})$.
\end{proof}

\begin{proof}[Proof of Lemma~\ref{lemma:Ap}]
$w(x',x_d)=\max(1,\,|x'|^d,\,|x_d+|x'|^2|^d)$.
Therefore by repeated applications of Lemma~\ref{lemma:maxofweights},
it suffices to prove that $|x'|^{td}\in A_p$
and $\big|x_d+|x'|^2\big|^{td}\in A_p$.

Let $s\ge 0$.
It is well known that $u(x')=|x'|^s$ belongs to $A_p(\reals^{d-1})$
if and only if $u^{-1/(p-1)}\in L^1_{\text{loc}}(\reals^{d-1})$,
thus if and only if $s<(p-1)(d-1)$.
It follows easily that $u(x',x_d)=|x'|^s$ belongs to
$A_p(\reals^{d})$ for the same range of $s$, that is, if and only if
$s<(p-1)(d-1)$.

It is elementary that $\big| x_d+|x'|^2 \big|^{s}\in A_p(\reals^d)$
whenever $s/(p-1)<1$ and $2s/(p-1)<d-1$; details are left to the reader.
Alternatively, a general result \cite{stein} p.\ 219 asserts that for any polynomial
$P$ of degree $D$, $|P|^s\in A_p(\reals^d)$ whenever $sD<p-1$;
in our case this implies that $\big| x_d+|x'|^2 \big|^s\in A_p$
for all $s < (p-1)/2$.
\end{proof}

\begin{remark}
The weight $w^{tp_t}$ therefore belongs to $A_q$
whenever $dtp_t<q-1$.
Substituting $p_t=(1-t)^{-1}p_0=(1-t)^{-1}(d+1)/d$ gives
the sufficient condition
\[
q > 1+dt(1-t)^{-1}(d+1)/d
= 1 + t(1-t)^{-1}(d+1)
\]
For $q=p_t$ this becomes
\[
p_t>1+dtp_t
\Leftrightarrow 1>p_t^{-1}+dt
\Leftrightarrow 1>(1-t)d/(d+1)+dt.
\]
This is clearly not satisfied for $t=1$, but is satisfied for $0\le t<1/d^2$.
\end{remark}

Denote by $|D|^r$ the differentiation operators
$\widehat{|D|^rf}(\xi) = |\xi|^r\widehat{f}(\xi)$.
We will use the notation $\abr{z}=(1+|z|^2)^{1/2}$ for $z\in\complex$.
Let $\nabla^*$ denote the divergence of a vector field.

\begin{lemma} \label{lemma:CZweighted}
There exists $C\in(0,\infty)$ such that
for all sufficiently small $t\ge 0$,
\begin{equation} \label{eq:CZweighted}
C^{-1} \norm{h}_{X_t}
\le
\norm{\nabla |D|^{-1}h}_{X_t}
\le C\norm{h}_{X_t}
\end{equation}
for all $h\in X_t$.
\end{lemma}

\begin{proof}
The operators
$\nabla\circ |D|^{-1}$ and $|D|^{-1}\circ\nabla^*$
are Calder\'on-Zygmund operators of classical type.
Moreover, $|D|^{-1}\nabla^*\circ \nabla |D|^{-1}$ is the identity operator.
The weight used to define $X_t$
belongs to $A_{p_t}$, provided that $t$ is sufficiently small.
Therefore \eqref{eq:CZweighted}
follows from the theory of weighted Calder\'on-Zygmund inequalities.
\end{proof}

\begin{lemma} \label{lemma:powersofD}
For all sufficiently small $\varrho\ge 0$,
\[
\norm{|D|^{1-\gamma}f}_{X_\varrho}
\le
\norm{\nabla f}_{X_\varrho}^{1-\gamma}
\norm{f}_{X_\varrho}^\gamma.
\]
\end{lemma}
\begin{proof}[Sketch of proof]
Consider the analytic family
of operators $z\mapsto |D|^z$.
For all sufficiently small $t\ge 0$,
$|D|^{i\sigma}$ is bounded on $X_t$ with a norm $\lesssim \abr{\sigma}^C$,
uniformly for all $\sigma\in\reals$. This inequality holds by
Lemma~\ref{lemma:Ap} and the theory of $A_p$ weighted inequalities
for Calder\'on-Zygmund operators.
It follows that
\[
\norm{|D|^{1-\gamma}f}_{X_\varrho}
\le
\norm{|D|f}_{X_\varrho}^{1-\gamma}
\norm{f}_{X_\varrho}^\gamma.
\]
But $|D|f$ may be replaced by $\nabla f$, by Lemma~\ref{lemma:CZweighted}.
\end{proof}

\section{Multlinear Bounds} \label{section:prelim}

Define the multilinear operators
\begin{equation}
\vec{S}(f_{i,j})_{i,j=1}^d
= \prod_{i=1}^d T^*(\prod_{j=1}^d T(f_{i,j})).
\end{equation}
Thus $Sf = \vec{S}(f,f,\cdots,f)$.
We will sometimes write this more simply as $\vec{S}(\vec{f})$
where $\vec{f}=(f_\alpha: \alpha\in A)$, with $A=\{1,2,\cdots,d\}^2$.
Repeated applications of H\"older's inequality lead to the inequality
\begin{equation} \label{CSforS}
|\vec{S}(\vec{f})|\le \Big(\prod_{\alpha\in\{1,2,\cdots,d\}^2}
S(|f_{\alpha}|)\Big)^{1/d^2}.
\end{equation}
Indeed, for any nonnegative functions $g_i$,
\begin{align*}
T^*\prod_{i=1}^d g_i
&=\int_{\reals^{d-1}}
\prod_{i=1}^d g_i(x'-t,x_d+|t|^2)\,dt
\\
&\le
\prod_{i=1}^d
\big(\int_{\reals^{d-1}}
g_i(x'-t,x_d+|t|^2)^d\,dt\big)^{1/d}
\\
&=
\prod_{i=1}^d
(T^*(g_i^d))^{1/d}.
\end{align*}
Applying this to $g_i=|T(f_{i,j})|\le T(|f_{i,j}|)$ gives
\eqref{CSforS}.

\begin{lemma} \label{lemma:holderforXt}
Let $A$ be any finite index set.
Let $\theta_{\alpha},t_\alpha\in[0,1]$ for each $\alpha\in A$.
Suppose that $\sum_{\alpha\in A}\theta_\alpha=1$
and
\begin{equation} \label{convexsum}
1-t = \sum_{\alpha\in A} \theta_\alpha (1-t_\alpha).
\end{equation}
Then for any nonnegative functions $f_\alpha$,
\begin{equation}
\norm{\prod_{\alpha\in A} f_\alpha^{\theta_\alpha}}_{X_t}
\le
\prod_{\alpha\in A} \norm{f_\alpha}_{X_{t_\alpha}}^{\theta_\alpha}.
\end{equation}
\end{lemma}

\begin{proof}
In this proof, products with respect to $\alpha$ are always taken over all $\alpha\in A$.
Recall that $p_t^{-1}=(1-t)p_0^{-1}$.
Introduce the exponents
\[q_\alpha^{-1} = \frac{\theta_\alpha(1-t_\alpha)}{1-t}
= \frac{\theta_\alpha p_t}{p_{t_\alpha}}\]
and
\[
r_\alpha = \frac{p_0 t_\alpha\theta_\alpha }{1-t}
= p_{t_\alpha}t_\alpha\theta_\alpha.
\]

Since
\begin{multline*}
\sum_\alpha r_\alpha
=p_0(1-t)^{-1} \sum_\alpha t_\alpha\theta_\alpha
=p_0(1-t)^{-1} \sum_\alpha (t_\alpha-1)\theta_\alpha
+p_0(1-t)^{-1} \sum_\alpha \theta_\alpha
\\
=p_0(1-t)^{-1}\cdot (t-1)
+p_0(1-t)^{-1}\cdot 1
=p_0t(1-t)^{-1}
=tp_t,
\end{multline*}
one can write
\begin{align*}
\norm{\prod_\alpha f_\alpha^{\theta_\alpha} }_{X_t}^{p_t}
&= \int_{\reals^d}
\prod_\alpha f_\alpha^{p_t \theta_\alpha} w^{tp_t}
\\
&= \int \prod_\alpha \Big( f_\alpha^{p_t \theta_\alpha} w^{r_\alpha}\Big).
\end{align*}
Since
$\sum_\alpha q_\alpha^{-1} = 1$
by the hypothesis \eqref{convexsum}, H\"older's inequality gives
\begin{align*}
\int \prod_\alpha \Big( f_\alpha^{p_t \theta_\alpha} w^{r_\alpha}\Big)
&\le
\prod_\alpha\Big(
\int f_\alpha^{q_\alpha p_t\theta_\alpha} w^{r_\alpha q_\alpha}
\Big)^{1/q_\alpha}.
\end{align*}

The exponents in this last expression can be simplified:
$q_\alpha p_t\theta_\alpha = p_{t_\alpha}$,
while
\[r_\alpha q_\alpha = p_0(1-t_\alpha)^{-1}t_\alpha
=t_\alpha p_{t_\alpha}.\]
Thus the last expression is simply
\[
\prod_\alpha\Big(
\int f_\alpha^{p_{t_\alpha}} w^{t_\alpha p_{t_\alpha}}
\Big)^{p_t \theta_\alpha/p_{t_\alpha} }
=
\prod_\alpha
\norm{f_\alpha}_{X_{t_\alpha}}^{\theta_\alpha p_t}.
\]
Raising everything to the power $1/p_t$ establishes the lemma.
\end{proof}

Now let $A=\{1,\dots,d\}^2$.
\begin{corollary} \label{cor:holderforXt}
Let $t\in[0,1/d^2]$.
Let $\beta\in A$.
Let $(f_\alpha: \alpha\in A)$ satisfy
$f_\beta\in X_{d^2 t}$ and
$f_\alpha\in X_0$ for all $\alpha\ne \beta$.
Then
$\vec{S}(\vec{f})\in X_t$,
and
\begin{equation}
\norm {\vec{S}(\vec{f})}_{X_t}
\le C
\norm{f_\beta}_{X_{d^2 t}}
\prod_{A\owns \alpha\ne\beta} \norm{f_\alpha}_{X_0}.
\end{equation}
\end{corollary}

\begin{proof}
$|\vec{S}(\vec{f})|\le \prod_\alpha S(|f_\alpha|)^{1/d^2}$.
Let $\theta_\alpha=1/d^2$ for all $\alpha\in A$.
Let $t_\alpha=0$ for all $\alpha\ne\beta$,
and $t_\beta = d^2 t$.
These parameters satisfy the hypotheses \eqref{convexsum}
of Lemma~\ref{lemma:holderforXt}.

Therefore
\begin{multline*}
\norm{\vec{S}(\vec{f})}_{X_t}
\le
\norm{ \prod_\alpha S(|f_\alpha|)^{1/d^2} }_{X_t}
\le \prod_\alpha \norm{ S(|f_\alpha|)}_{X_{t_\alpha}}^{\theta_\alpha}
\\
\le C\prod_\alpha \norm{ f_\alpha}_{X_{t_\alpha}}^{d^2\theta_\alpha}
= C
\norm{f_\beta}_{X_{d^2t}}
\prod_{\alpha\ne\beta} \norm{ f_\alpha}_{X_0}.
\end{multline*}
\end{proof}

\section{Smoothing}

Consider the operators
$T_\rho f(x) = \int_{|t|\le \rho} f(x'-t,x_d-|t|^2)\,dt$.

\begin{lemma} \label{lemma:locallygainsmoothness}
There exists $\alpha_0>0$ such that for any
$\alpha\in[0,\alpha_0]$,
there exist $C,A<\infty$
such that for all $\rho\in[1,\infty)$,
\[
\norm{|D|^\alpha T_\rho f}_{L^2(\reals^d)}
\le C\rho^A\norm{f}_{L^2(\reals^d)}
\]
for all $f\in\lt(\reals^d)$.

For any $p\in(1,\infty)$ there exist $\eta>0$
and $C,A<\infty$
such that for all $\rho\in[1,\infty)$,
\[
\norm{|D|^\eta T_\rho f}_{L^p(\reals^d)}
\le C\rho^A\norm{f}_{L^p(\reals^d)}
\]
for all $f\in L^p(\reals^d)$.

\end{lemma}

The proof of the first conclusion
is a routine application of the method of stationary phase.
See \cite{stein} for calculations of this type.

$T_\rho$ is defined by convolution with a finite measure of total
variation $O(\rho^d)$, and consequently
satisfies $\norm{T_\rho f}_{L^p}\le C\rho^{d-1}\norm{f}_{L^p}$
for all $p\in[1,\infty]$.
The second conclusion is by interpolating between these simple bounds
and the first conclusion, using an analytic family of operators
$z\mapsto |D|^z\circ T_\rho$.
\qed

Let $B_R=\{x\in\reals^d: |x|\le R\}$.
\begin{corollary} \label{cor:smoothing}
Let $t>0$. There exists $\gamma=\gamma(t)>0$
such that for any $f\in X_t$,
$|D|^\gamma(Tf)\in L^{q_0}_{\text{loc}}$.
More quantitatively,
for any $R<\infty$ there exists $C<\infty$
such that for any $f\in X_t$,
\[
\norm{|D|^\gamma(Tf)}_{L^{q_0}(B_R)}
\le C\norm{f}_{X_t}.
\]
\end{corollary}

\begin{proof}
There exists $\delta>0$ such that
whenever $|x|\ge 1$, $w(x)\ge c|x|^{d/2}$.
Indeed, \begin{multline*}
w(x) = \max(1,\,|x'|^d,\,|x_d-|x'|^2|^d)
\\
\gtrsim
\max(\abr{x'},\abr{x_d-|x'|^2}^{1/2})^d
\ge c \max(\abr{x'},\abr{x_d}^{1/2})^d.
\end{multline*}

Let $1\le R,\rho<\infty$.
Define $T^\natural f$ to be the restriction
of $Tf$ to $B(0,R)$.
Let $t>0$ and
consider any function $g\in X_t$ supported in $B(0,\rho)\setminus B(0,\rho/2)$
satisfying $\norm{g}_{X_t}\le 1$.

Then $\norm{g}_{X_{t/2}}=O(\rho^{-\delta})$,
where $\delta>0$ depends only on $t$.
Therefore $\norm{Tg}_{Y_*,{t/2}}\lesssim \rho^{-\delta}$.
The space $Y_{*,t/2}$ embeds continuously into $L^{q_1}$
for some $q_1>q_0$, yielding
\begin{equation} \label{smoothing1}
\norm{T^\natural g}_{L^{q_1}}\lesssim \rho^{-\delta}.
\end{equation}

On the other hand,
because $g$ is supported in $B(0,\rho)$
and $T^\natural g$ is the restriction of $Tg$ to $B(0,R)$,
$T^\natural g$ equals the restriction of $T_s g$ to $B(0,R)$
where $s=C(R+\rho)\le CR\rho$. Therefore
by Lemma~\ref{lemma:locallygainsmoothness},
there exists $\eta>0$ such that
\begin{equation} \label{smoothing2}
\norm{|D|^\eta T^\natural g}_{L^{p_0}}
\le CR^A\rho^A \norm{g}_{L^{p_0}}
\le CR^A\rho^A \norm{g}_{X_t}
\le CR^A\rho^A
\end{equation}
for a certain finite constant $A$, which depends only on the dimension $d$.

By interpolating between  \eqref{smoothing1} and \eqref{smoothing2}
using the natural analytic family of operators,
we find that for any $\theta\in [0,1]$,
\begin{equation}
\norm{|D|^{\eta\theta} T^\natural g}_{L^{Q(\theta)}}
\le CR^{A\theta} \rho^{A\theta-(1-\theta)\delta}
\end{equation}
where $Q(\theta)^{-1} = \tfrac12 \theta + \frac1{q_1}(1-\theta)$.
Then $Q(0)^{-1}=q_1^{-1}<q_0^{-1}$.
Therefore for all sufficiently small
$\theta>0$, $A\theta-(1-\theta)\delta<0$
and
$Q(\theta)^{-1}<q_0^{-1}$.
Fix one such parameter $\theta$.
By H\"older's inequality,
\[
\norm{|D|^{\eta\theta}T^\natural g}_{L^{q_0}}
\le CR^C
\norm{|D|^{\eta\theta}T^\natural g}_{L^{Q(\theta)}}.
\]
Therefore in all,
\[
\norm{|D|^{\eta\theta} T^\natural g}_{L^{q_0}}
\le CR^{C} \rho^{-\eps}
\]
for some $C<\infty$ and $\eps,\theta,\eta>0$.

We have proved that
\begin{equation} \label{smoothing3}
\norm{|D|^{\gamma} Tg}_{L^{q_0}(B(0,R)}
\le CR^{C} \rho^{-\eps},
\end{equation}
provided that $g$ is supported on $B(0,\rho)\setminus B(0,\rho/2)$.
Here $\eps,\gamma>0$.
The same reasoning gives
\begin{equation} \label{smoothing4}
\norm{|D|^{\gamma} Tg}_{L^{q_0}(B(0,R)}
\le CR^{C} \rho^C
\end{equation}
if $g$ is merely assumed to be supported on $B(0,\rho)$.

The proof of Corollary~\ref{cor:smoothing}
is concluded by decomposing a general function $f$
as $\sum_{k=0}^\infty f_k$ where $f_0$ is supported
on $B(0,R)$ and $f_k$ on $B(0,2^k)\setminus B(0,2^{k-1})$
for all $k\ge 1$. Apply \eqref{smoothing3} to the contribution of
$f_k$ for all $k\ge 1$, and \eqref{smoothing4} for $k=0$,
and sum over $k$.
\end{proof}

\section{Gaining some decay}
Our goal here is to prove:
\begin{proposition} \label{prop:gaindecay}
Let $d\ge 2$ and $\lambda\in\complex$. Let $f\in L^{p_0}(\reals^d)=X_0(\reals^d)$
be a solution of the generalized Euler-Lagrange equation $f = \lambda Sf$.
Then there exists $t>0$ such that $f\in X_t$.
\end{proposition}

To begin the proof,
consider any decomposition $f=\varphi+g$
where $\varphi\in L^\infty$ has bounded support.
Rewrite the equation $f=\lambda Sf$ as
\begin{align}
g &= \lambda Sg + \scriptl(\varphi,g)
\\
\scriptl(\varphi,g) &= \lambda S(\varphi+g)-\lambda Sg-\varphi.
\label{scriptldefn}
\end{align}
Then
\[
\norm{\scriptl(\varphi,g)}_{X_t}
\le
\norm{g}_{X_t} + C\norm{g}_{X_t}^{d^2}
\]
by the representation $\scriptl(\varphi,g)=g-\lambda Sg$
and the basic $X_t$ bound for $S$.
On the other hand, by expanding $S(\varphi+g)$ as a sum of
$d^2$ terms $\vec{S}(\cdot)$ and invoking \eqref{CSforS} along with
the bound $\norm{S(h)}_{X_t}\le C\norm{h}_{X_t}^{d^2}$
gives an alternative majorization
\[
\norm{\scriptl(\varphi,g)}_{X_t}
\le C_\varphi + C_\varphi \norm{g}_{X_t}^{d^2-1}.
\]

This bound can be improved;
the operator $g\mapsto\scriptl(\varphi,g)$ improves integrability in
the following sense.
\begin{lemma} \label{lemma:improvespace}
For any bounded, compactly supported function $\varphi$
there exists $C_\varphi<\infty$ such that for all $g\in X_0$,
the function $\scriptl(\varphi,g)$ belongs to $X_{1/d^2}$, and
\begin{equation}
\norm{\scriptl(\varphi,g)}_{X_{1/d^2}}
\le
C_\varphi
+
C_\varphi \norm{g}_{X_0}^{d^2-1}.
\end{equation}
\end{lemma}

\begin{proof}
By assumption, $\varphi\in X_1\subset X_{1/d^2}$, so
it suffices to show that $S(\varphi+g)-S(g)$ satisfies the required bound.
Let $A=\{1,2,\cdots,d\}^2$.
$S(\varphi+g)-S(g)$ can be expanded as a sum of $d^2-1$ terms,
each of which is of the general form $\vec{S}(\vec{f})$
where $\vec{f}=(f_\alpha: \alpha\in A)$,
where each $f_\alpha$ equals either $\varphi$ or $g$,
and where for each such term, there exists at least one index $\beta\in A$
for which $f_\beta=\varphi$.
The required bound therefore follows directly from Corollary~\ref{cor:holderforXt},
again since $\varphi\in X_1$.
\end{proof}

We continue with the proof of Proposition~\ref{prop:gaindecay}.
Let $\eps>0$. Decompose $f=\varphi_\eps+g_\eps$
where $\norm{g_\eps}_{X_0}<\eps$,
and $\varphi_\eps\in L^\infty$ has bounded support.
Define
\[A_\eps(h) = \lambda Sh + \scriptl(\varphi_\eps,g_\eps).\]
This operator depends of course on $\varphi_\eps,g_\eps$,
and is defined in such a way that $A_\eps(g_\eps)=g_\eps$,
that is, $g_\eps$ is one solution of the fixed point
equation $A_\eps(h)=h$ in the space $X_0$.

\begin{lemma} \label{lemma:fixedpoint}
Let $\lambda\in\complex$, and
let $f\in L^{p_0(d)}(\reals^d)$ be any solution of $f=\lambda S(f)$.
For each $\eps>0$, let $f=\varphi_\eps+g_\eps$ be any decomposition
with $\varphi_\eps$ bounded and having bounded support,
and with $\norm{g_\eps}_{L^{p_0}}<\eps$.
Then there exists $\eps_0>0$ such that
for each $\eps\in(0,\eps_0]$,
there exists $t_\eps>0$ such that for all $t\in[0,t_\eps]$,
the fixed point equation
\[
A_\eps(h)=h
\]
has a unique solution $h\in X_{t}$
satisfying $\norm{h}_{X_{t}}\le\eps^{1/2}$.
\end{lemma}
It bears emphasis that there are no {\it a priori}
bounds for $\eps_0$ or $t_\eps$; these depend on $f$
in some uncontrolled manner.

\begin{proof}
We know that
\[\norm{\scriptl(\varphi_\eps,g_\eps)}_{X_0} \le \eps+C\eps^{d^2},\]
and that
\[\scriptl(\varphi_\eps,g_\eps)\in X_{1/d^2}.\]
By convexity of the $X_t$ norms,
for each sufficiently small $\eps>0$
there exists $t_\eps>0$ such that
\[\norm{\scriptl(\varphi_\eps,g_\eps)}_{X_{t_\eps}} \le \eps^{3/4}.\]
Henceforth we consider only such small $\eps$.

Let $B_\eps$ be the ball of radius $\eps^{1/2}$ in $X_{t_\eps}$,
centered at $0$.
If $h\in B_\eps$ then
\begin{align*}
\norm{A_\eps(h)}_{X_{t_\eps}}
&\le |\lambda|\cdot\norm{Sh}_{X_{t_\eps}}
+ \norm{\scriptl(\varphi_\eps,g_\eps)}_{X_{t_\eps}}
\\
&\le C\norm{h}_{X_{t_\eps}}^{d^2}
+ \eps^{3/4}
\\
&\le C\eps^{d^2/2} + \eps^{3/4}
\\
&<\eps^{1/2},
\end{align*}
so $A_\eps(B_\eps)\subset B_\eps$.
For any $h,\tilde h\in B_\eps$,
\begin{align*}
\norm{A_\eps(h)-A_\eps(\tilde h)}_{X_{t_\eps}}
&=
|\lambda|\cdot \norm{Sh-S\tilde h}_{X_{t_\eps}}
\\
&\le C\norm{h-\tilde h}_{X_{t_\eps}}\cdot
\Big(
\norm{h}_{X_{t_\eps}}
+
\norm{\tilde h}_{X_{t_\eps}}
\Big)^{d^2-1}
\\
&\le C\eps^{1/2}
\norm{h-\tilde h}_{X_{t_\eps}}.
\end{align*}
Therefore $A_\eps:B_\eps\to B_\eps$ is a strict contraction,
for each sufficiently small $\eps$.
Therefore there exists a unique $h_\eps\in X_{t_\eps}$
satisfying both $\norm{h_\eps}_{X_{t_\eps}}\le\eps^{1/2}$
and $A_\eps(h_\eps)=h_\eps$.

Exactly the same reasoning applies in $X_t$ for any $0\le t\le t_\eps$.
\end{proof}

\begin{proof}[Proof of Proposition~\ref{prop:gaindecay}]
Suppose that $0\le s\le t\le t_\eps$,
and that both $h\in X_{s}$ and $\tilde h\in X_{t}$
are solutions of $A_\eps(h)=h$,
satisfying $\norm{h}_{X_{s}}\le \eps^{1/2}$
and $\norm{\tilde h}_{X_{t}}\le \eps^{1/2}$.
Then
\[
\norm{\tilde h}_{X_s}=\norm{A_\eps\tilde h}_{X_s}
\le C\norm{A_\eps\tilde h}_{X_t}\le C\eps^{3/4}<\eps^{1/2},
\]
provided that $\eps$ remains sufficiently small.
Therefore $\tilde h=h$ by the uniqueness of solutions.

In particular, since
$g_\eps$
is a solution in $X_0$,
this uniqueness of solutions implies that
\[g_\eps=h_\eps\in X_{t_\eps}\]
for all sufficiently small $\eps>0$,
as was to be proved.
\end{proof}

\section{Smoothness} \label{section:smoothness}
We have shown that any solution of the Euler-Lagrange equation
enjoys some extra decay, beyond that encoded by the finiteness of its $L^{p_0}$ norm.
We will next show how such extra decay can be used
in conjunction with the Euler-Lagrange equation to demonstrate some smoothness.
Our initial goal is to prove the following {\em a priori} inequality.

\begin{lemma} \label{lemma:apriori}
Let $\rho>0$ be sufficiently small.
Then for any $0\le\varrho<\rho$
there exists $C<\infty$ such that
for any solution $f$ of $f=\lambda Sf$,
if $f\in X_\rho$ and $\nabla f\in X_\varrho$ then
\begin{equation} \label{ineq:apriori}
\norm{\nabla f}_{X_\varrho} \le C\norm{f}_{X_\rho}^{d^2}.
\end{equation}
Here $C$ depends only on $\rho,\varrho,\lambda,d$.
\end{lemma}

Here $\nabla f
=\big(\frac{\partial f}{\partial x_1},\cdots,\frac{\partial f}{\partial x_d}\big)$.
It suffices to prove this under the assumption that $\norm{f}_{X_\rho}=1$,
which will be assumed for the remainder of \S\ref{section:smoothness}.
Indeed, for general $f$,
consider the function $F = f/\norm{f}_{X_\rho}$.
It satisfies the modified equation
$F = \tilde\lambda SF$
where $\tilde\lambda = \lambda \norm{f}_{X_\rho}^\sigma$
for a certain exponent $\sigma$.
Thus we only have to replace $\lambda$ by $\tilde\lambda$
in order to assume $\norm{f}_{X_\rho}=1$.

\begin{lemma} \label{lemma:apriori1}
Let $\rho>0$ be sufficiently small and $\lambda\in\complex$.
Let $0<\varrho<\rho$.
There exists $R<\infty$
such that for any function $f$
satisfying $f=\lambda Sf$ and
$\norm{f}_{X_\rho}=1$, with $\nabla f\in X_\varrho$,
\[
\norm{\nabla f}_{X_\varrho}
\le C
\norm{T\nabla f}_{L^{q_0}(B_R)}
\]
where $C,R,a$ depend only on $d,\rho,\varrho,\lambda,\norm{f}_{X_\rho}$.
\end{lemma}

\begin{proof}[Proof of Lemma~\ref{lemma:apriori1}]
Write
\begin{align}
\nabla f = \lambda\nabla(Sf)
&=d^2\lambda\vec{S}(f,f,\cdots,f,\nabla f)
\notag
\\
&=d^2\lambda \big(T^*([Tf]^{d})\big)^{d-1}
\cdot T^*([Tf]^{d-1}\cdot\nabla Tf).
\label{nablafrep}
\end{align}
Here $\vec{S}(f,f,\cdots,f,\nabla f)$ stands for the vector with $d$
components, whose $j$-th component equals
$\vec{S}(f,f,\cdots,f,\partial f/\partial x_j)$.

Therefore
\begin{equation}
\norm{\nabla f}_{X_\varrho}
\le C\norm{f}_{X_\rho}^{d^2-1}\norm{\nabla Tf}_{Y_{*,\varrho'}}
= C\norm{\nabla Tf}_{Y_{*,\varrho'}}
\end{equation}
for a certain $\varrho'<\varrho$; $\varrho'$ does not depend on $f$.
Since
\begin{align*}
\norm{\nabla Tf}_{Y_{*,\varrho'}}
&\le\norm{\nabla Tf}_{Y_{*,\varrho}}^\theta
\norm{\nabla Tf}_{Y_{*,0}}^{1-\theta}
\\
&=\norm{\nabla Tf}_{Y_{*,\varrho}}^\theta
\norm{\nabla Tf}_{L^{q_0}}^{1-\theta}
\\
&\lesssim\norm{\nabla f}_{X_{\varrho}}^\theta
\norm{\nabla Tf}_{L^{q_0}}^{1-\theta}
\end{align*}
for some $\theta\in (0,1)$,
we deduce that
\begin{equation} \label{readyforabsorption}
\norm{\nabla f}_{X_\varrho}
\le C\norm{\nabla Tf}_{L^{q_0}}
\end{equation}
where $C$ depends only on $\rho,\varrho, \lambda$.
Now for any $R<\infty$,
\begin{align*}
\int_{|x|\ge R}|\nabla Tf(x)|^{q_0}\,dx
&\le R^{-q_0\tau\varrho} \norm{\nabla Tf}_{Y_{*,\varrho}}^{q_0}
\\
&= R^{-q_0\tau\varrho} \norm{T\nabla f}_{Y_{*,\varrho}}^{q_0}
\\
&\le C R^{-q_0\tau\varrho} \norm{\nabla f}_{X_{\varrho}}^{q_0}
\end{align*}
for a certain exponent $\tau>0$.
Therefore
\begin{align*}
\norm{\nabla f}_{X_\varrho}
&\le
C \norm{\nabla Tf}_{L^{q_0}(B_R)}
+ C R^{-\tau\varrho}
\norm{\nabla f}_{X_\varrho}.
\end{align*}
Define $R$ by the equation
$C R^{-\tau\varrho} =\tfrac12$
to obtain
\begin{equation} \label{notquiteneeded}
\norm{\nabla f}_{X_\varrho}
\le
2C
\norm{\nabla f}_{L^{p_0}(B_R)}
\end{equation}
where
$C,R$ depend only on $\rho,\varrho,\lambda,d$.
$R$ will henceforth remain fixed.
This same reasoning can be carried out for all dimensions $d$ with very minor changes.
\end{proof}

\begin{proof}[Proof of Lemma~\ref{lemma:apriori}]
We will use the representation \eqref{nablafrep} in order to
obtain a bound for $\norm{T\nabla f}_{L^{q_0}(B_R)}$
in terms of $\norm{f}_{X_\rho}$, where $R$ is as defined above.


Let $\gamma\in(0,1)$ be a small constant, to be chosen below.
Writing \[T\nabla f=|D|^\gamma T(\nabla |D|^{-\gamma}f),\]
Corollary~\ref{cor:smoothing} gives
\[
\norm{T\nabla f}_{L^{q_0}(B_R)}
=\norm{|D|^\gamma T\big(\nabla |D|^{-\gamma}f\big)}_{L^{q_0}(B_R)}
\lesssim
\norm{\nabla |D|^{-\gamma}f}_{X_\varrho}
\]
provided that $\gamma$ is a sufficiently small function of $\varrho,d$ alone.


Therefore by Lemma~\ref{lemma:CZweighted},
Lemma~\ref{lemma:powersofD},
and a second application of Lemma~\ref{lemma:CZweighted},
\begin{align*}
\norm{\nabla |D|^{-\gamma}f}_{X_\varrho}
\le C
\norm{|D|^{1-\gamma}f}_{X_\varrho}
\le C
\norm{\nabla f}_{X_\varrho}^{1-\gamma}
\norm{f}_{X_\varrho}^\gamma
\end{align*}
for some $\gamma\in(0,1)$.

Thus
\[
\norm{T\nabla f}_{L^{q_0}(B_R)}
\lesssim \norm{\nabla f}_{X_\varrho}^{1-\gamma}\norm{f}_{X_\rho}^\gamma
= \norm{\nabla f}_{X_\varrho}^{1-\gamma}
\]
and therefore by Lemma~\ref{lemma:apriori1},
\[
\norm{\nabla f}_{X_\varrho}
\lesssim
\norm{\nabla f}_{X_\varrho}^{1-\gamma}.
\]
Recall that
$\norm{\nabla f}_{X_\varrho}^{1-\gamma}$
is assumed to be finite. It follows from this last inequality that
$\norm{\nabla f}_{X_\varrho} \lesssim 1$.
This completes the proof of Lemma~\ref{lemma:apriori}.
\end{proof}

\section{Mollified Derivatives and Conclusion of Proof}
Lemma~\ref{lemma:apriori}
presupposes that $\nabla f\in X_\rho$, which we seek to prove.
In order to remove the extraneous assumption,
we approximate $\nabla$ by a one-parameter family of
operators which are individually bounded on the spaces $X_t$.

For any $s\ge 0$ and $\Lambda\ge 1$ define
\begin{equation}
\widehat{D^s_\Lambda f}(\xi) =
\min\big(1+|\xi|^2)^{1/2},(1+\Lambda^2)^{1/2}\big)^s \cdot \widehat{f}(\xi).
\end{equation}
These operators are bounded on all $L^p$ spaces,
and likewise on all spaces $X_t$ for $t\in[0,1]$.
For $s=1$ we write simply $D_\Lambda$.

In order to prove that $\nabla f\in X_\varrho$,
it suffices to show that $\norm{D_\Lambda f}_{X_\varrho}\le A$
for some finite constant $A$ which is independent of $\Lambda$.
The proof of Lemma~\ref{lemma:apriori} relied on Leibniz' rule for derivatives
of products. There is no corresponding formula for $D_\Lambda(fg)$, but
the following lemma provides an adequate substitute.

\begin{lemma} \label{lemma:leibniz}
Let $u\ge 0$ be a locally integrable function.
Let $s\in(0,\infty)$.
Suppose that $r^{-1} = p_j^{-1}+q_j^{-1}$
for $j=1,2$ and that all exponents $r,p_j,q_j$
belong to the open interval $(1,\infty)$.
Suppose that the weight $u$ belongs to $A_r$
and that $u=u_1v_1=u_2v_2$
where $u_j^{p_j/r}\in A_{p_j}$
and
$v_j^{q_j/r}\in A_{q_j}$.
Then
there exists $C<\infty$ such that
$D^s_\Lambda(fg)\in L^r$, and the following inequality holds,
whenever the right-hand side is finite:
\begin{equation}
\norm{D^s_\Lambda (fg)}_{L^r(u)}
\le C
\norm{D^s_\Lambda f}_{L^{p_1}(u_1^{p_1/r})} \norm{g}_{L^{q_1}(v_1^{q_1/r})}
+ C
\norm{f}_{L^{p_2}(u_2^{p_2/r})}
\norm{D^s_\Lambda g}_{L^{q_2}(v_2^{q_2/r})}.
\end{equation}
\end{lemma}

A proof will be given in \S\ref{section:leibniz}.

\begin{corollary} \label{cor:leibniz}
Let $s\in(0,\infty)$.
Let $\rho>0$ be sufficiently small, and let $0<\varrho<\rho$.
There exist $\varrho'\in(0,\varrho)$
and $C<\infty$ such that
for all $\Lambda\ge 1$
and  all vector-valued functions $\vec{f}\in X_\rho$,
\[
\norm{D^s_\Lambda\vec{S}(\vec{f})}_{X_\varrho}
\le C\sum_j\prod_{i\ne j}
\norm{f_i}_{X_\rho}
\cdot \norm{D^s_\Lambda Tf_j}_{Y_{*,\varrho'}}.
\]
The constant $C$ may be taken to be
independent of $\Lambda$ while $s,\rho$ remain fixed.
\end{corollary}

\medskip
Together, the proof of Lemma~\ref{lemma:apriori} and Corollary~\ref{cor:leibniz}
establish:
\begin{lemma}
Let $s\in(0,\infty)$ and $\lambda\in\complex$.
Let $\rho>0$ be sufficiently small, and let $0<\varrho<\rho$.
Let $f\in X_\rho$ be any solution
of the Euler-Lagrange equation \eqref{eq:EL}.
There exists $C<\infty$
such that for all $\Lambda<\infty$,
\[
\norm{D^s_\Lambda f}_{X_\varrho}\le C.
\]
\end{lemma}

Because this bound is uniform in $\Lambda$, combining this lemma
with Proposition~\ref{prop:gaindecay} yields:
\begin{corollary}
Let $\lambda\in\complex$.
Let $f\in X_\rho$
be any solution of the Euler-Lagrange equation \eqref{eq:EL}.
Then there exists $t>0$ such that for all $s\ge 0$,
$|D|^s f\in X_t$.
\end{corollary}

It is now an easy consequence of Sobolev embedding that any solution
of \eqref{eq:EL} is $C^\infty$, completing the proof of Theorem~\ref{thm:main}.
\qed

\section{Proof of Lemma~\ref{lemma:key}} \label{section:tedious}

Recall the definitions
\begin{align*}
T^*f(x)&=\int_{\reals^{d-1} } f(x'+t,x_d+|t|^2)\,dt
\\
\upsilon(x) &= \min\big(1,\,|x'|^{-d},\,|x_d-|x'|^2|^{-d}\big)
\\
\upsilon_*(x) &= \min\big(1,\,|x'|^{-d},\,|x_d+|x'|^2|^{-d}\big)
\end{align*}
where $x\in\reals^d$ as written as $x=(x',x_d)\in\reals^{d-1}\times\reals^1$.
Lemma~\ref{lemma:key}
states that $T^*\upsilon\lesssim \upsilon_*^{1/d}$,
with a corresponding inequality for $T$.

Each of the weights $\upsilon,\upsilon_*$
is equal to a minimum of three functions having
three different degrees of homogeneity
$0,1,2$ with respect to the parabolic dilation group $x\mapsto (rx',r^2x_d)$,
so there is no dilation invariance to simplify the analysis.
Viewing $T^*\upsilon(x)$ as an integral with respect to a second
variable $y\in\reals^d$, and comparing the result to $\upsilon_*(x)^{1/d}$,
the estimation of $T^*\upsilon(x)$ splits naturally into $3\times 3=9$ cases.
This factor of $9$ accounts largely for the length of the proof which we now present;
in actuality some cases are combinable, but various subcases also arise.

\begin{proof}[Proof of Lemma~\ref{lemma:key}]
The two conclusions of Lemma~\ref{lemma:key}
can be shown to be equivalent by the change
of variables $(x',x_d)\mapsto (x',-x_d)$,
along with the substitution $t\mapsto -t$ in the integrals defining
$T,T^*$. So we will prove only the inequality
$T^*(\upsilon)\lesssim \upsilon_*^{1/d}$.

Write
\begin{multline*}
T^*\upsilon(x)=\int_{\reals^{d-1} } \upsilon(x'+t,x_d+|t|^2)\,dt
=\int_{\reals^{d-1} }\upsilon(s,x_d+|s-x'|^2)\,ds
\\
=
\int_{\reals^{d-1} }
\min\big(1,\,|s|^{-d},\,|x_d+|s-x'|^2-|s|^2|^{-d}\big)\,ds.
\end{multline*}
Observe that
\[
T^* \upsilon (x)
\lesssim
\int_{\reals^{d-1}} \abr{s}^{-d}\,ds\lesssim 1 \text{ uniformly for all $x\in\reals^{d-1}$}.
\]
This satisfies the required bound $C\upsilon_*^{1/d}(x)$
provided that $\upsilon_*(x)$ remains uniformly bounded below.
Therefore we may assume throughout the rest of the analysis
of the contribution of $T^*\upsilon(x)$ that
\begin{equation} \label{large}
\max(|x'|,|x_d+x'|^2)\gg 1.
\end{equation}

In the same way, because the integrand is $\le|s|^{-d}$
and because $\int_{|s|\ge\lambda}|s|^{-d}\,dx
\lesssim \lambda^{-1}$,
the contribution made to the integral by the set of all $s$ satisfying
$|s|\ge \tfrac14 \max(|x'|,|x_d+|x'|^2|)$
is
\[
\lesssim
\max(|x'|,|x_d+|x'|^2|)^{-1}
=\upsilon_*(x)^{1/d}.
\]
It remains to discuss the contribution of those $s$
which satisfy
\begin{equation} \label{0.5}
|s|< \tfrac14 \max(|x'|,|x_d+|x'|^2|).
\end{equation}

For each $x\in\reals^d$,
partition the set of all such $s\in\reals^{d-1}$ into two regions
\begin{equation}
\scriptr_1(x)=\{s\in \reals^{d-1}:
|s|\ge|x_d+|x'|^2-2x'\cdot s|
\ \text{ and } \ |s|< \tfrac14 \max(|x'|,|x_d+|x'|^2|)
\};
\end{equation} \begin{equation}
\scriptr_2(x)=\{s\in \reals^{d-1}:
|s|<|x_d+|x'|^2-2x' \cdot s|
\ \text{ and } \ |s|< \tfrac14 \max(|x'|,|x_d+|x'|^2|)\}.
\end{equation}
Thus we have shown that
\begin{equation}
T^*\upsilon(x) \le C\upsilon_*(x)^{1/d} + J_1(x)+J_2(x)
\end{equation}
where
\begin{equation*}
J_i(x)=\int_{\scriptr_i(x)}
\min\big(1,\,|s|^{-d},\,|x_d+|s-x'|^2-|s|^2|^{-d}\big)\,ds.
\end{equation*}
More succinctly,
\begin{align*}
&J_1(x)
\asymp \int_{\scriptr_1(x)}
\abr{s}^{-d}\,ds
\\
&J_2(x)
\asymp \int_{\scriptr_2(x)}
\abr{x_d+|x'|^2-2x'\cdot s}^{-d}\,ds.
\end{align*}
We will often write $J_i,\scriptr_i$ as shorthand for $J_i(x),\scriptr_i(x)$.

\noindent
\bull
\textbf {Estimate for $J_1(x)$ in the case $|x'|\ge |x_d+|x'|^2|$}.
In this case,
\[|x_d+|x'|^2-2x'\cdot s|\le |s|\le \tfrac14\max(|x'|,|x_d+|x'|^2|) < |x'|\]
by definition of $\scriptr_1(x)$,
so one of the following two subcases occurs:
\begin{equation}\label{0.6}
|x_d+|x'|^2-2x'\cdot s|\le 1
\ \text{ or } \
1<|x_d+|x'|^2-2x'\cdot s|\le |x'|.
\end{equation}

Consider first the contribution made to $J_1(x)$ by those
$s\in\scriptr_{1}$ which satisfy the first case in \eqref{0.6}.
There exists at least one index $i\in\{1,2,\cdots,d-1\}$
such that $|x_{i}|\ge {|x'|}/\sqrt{d-1}$.
Our problem is invariant with respect to rotations of $\reals^{d-1}$,
which leave the coordinate $x_d$ unchanged.
Therefore without loss of generality, we may assume
throughout the remainder of the proof of the Lemma that
\begin{equation}|x_{1}|\ge {|x'|}/\sqrt{d-1}.\end{equation}
We are working in the situation where
$1\ll \max(|x'|,|x_d+|x'|^2|)=|x'|\lesssim |x_1|$
by \eqref{large} and the definition of $\scriptr_{1,1}$,
so $|x_1|\gg 1$.

Introduce the notations
 \[ \tilde s = (s_2,\cdots,s_{d-1}) \ \text{ and }\
\phi(x,\tilde s)  =x_d+|x'|^2-2\sum_{i=2}^{d-1}x_is_i.  \]
Thus
$x_d+|x'|^2-2x'\cdot s=\phi(x,\tilde s)-2x_1s_1$,
so
$|2x_1s_1-\phi(x,\tilde s)|\le 1$
by definition of $\scriptr_{1}(x)$ and the first case of \eqref{0.6}.
The following fact will be used repeatedly throughout the analysis:
If $(x,\tilde s)$ is fixed, then an inequality
$|2x_1s_1-\phi(x,\tilde s)|\le \delta$
forces $s_1$ to lie in an interval
of length $\delta|x_1|^{-1}$.

Now the contribution made by those $s$ belonging to the first subcase of \eqref{0.6}
to
$\int_{\scriptr_{1}(x)}\abr{s}^{-d}\,ds$ is
\begin{multline*}
\lesssim
\int_{\scriptr_{1}(x)}\abr{s}^{-d}\,ds
\le \int_{\reals^{d-2}} \int_{|s_1-\phi(x,\tilde s)|\lesssim |x_1|^{-1}}
\abr{\tilde s}^{-d}\,ds_1\,d\tilde s
\lesssim
|x_1|^{-1}
\int_{\reals^{d-2}}
\abr{\tilde s}^{-d}\,d\tilde s
\lesssim |x_1|^{-1}.
\end{multline*}
This is the required bound, for $|x_1|^{-1}\lesssim \upsilon_*(x)^{1/d}$
because we are working in the case where
$|x_1|\gtrsim |x'|\ge |x_d+|x'|^2|\gg 1$.

Next we consider the contribution of those $s\in\scriptr_{1}(x)$ which satisfy
the second case in \eqref{0.6}, still under the assumption that
$|x_d+|x'|^2|\le|x'|$.
For $j,k\ge 1$ define
\[
E_1^{j,k}(x)
=\{s\in \scriptr_{1}(x):  2^{-k}<\frac{|x_d+|x'|^2-2x'\cdot s|}{|x'|}
\le 2^{-k+1}
\text{ and }
2^{-j}<\frac{|s|}{|x'|}\le 2^{-j+1} \}.
\]
For any $s\in E_1^{j,k}(x)$,
\[2^{-k}|x'|\le |x_d+|x'|^2-2x'\cdot s| \le |s|\le 2^{-j+1}|x'|.\]
Thus $j\le k +1$.

If $\tilde s$ remains fixed and $s\in E_1^{j,k}$,
then $s_1$ lies in an interval
of length $\asymp |x_1|^{-1}2^{-k}|x'|\asymp 2^{-k}$.
Since $|\tilde s|\le |s|\le 2^{-j+1}|x'|$,
\[|E_1^{j,k}(x)|\lesssim (2^{-j}|x'|)^{d-2}\cdot 2^{-k}.\]
Therefore the contribution made to $J_1(x)$ by all $s$ belonging to this subcase is
\begin{multline*}
\le\sum_{k=1}^{[\log_2|x'|+1]}\sum_{j=1}^{k+1}\int_{E_1^{j,k}(x)}  {\abr{ {s}}^{-d}}  \,ds
\le \sum_{k=1}^{[\log_2|x'|+1]}\sum_{j=1}^{k+1}
(2^{-j}|x'|)^{-d} |E_1^{j,k}(x)|
\\
\lesssim \sum_{k=1}^{[\log_2|x'|+1]}\sum_{j=1}^{k+1}(2^{-j}|x'|)^{-d}(2^{-j}|x'|)^{d-2} 2^{-k}
\\
=
|x'|^{-2}\sum_{k=1}^{[\log_2|x'|+1]}\sum_{j=1}^{k+1}
2^{2j}2^{-k}
\lesssim |x'|^{-1}.
\end{multline*}
This completes the analysis of $J_1(x)$
in the case where $|x'|\ge \big|x_d+|x'|^2 \big|$.

\noindent
\bull
\textbf {Estimate for $J_1(x)$ in the case $|x'|\le |x_d+|x'|^2|$}.
The definition \eqref{0.5} of $\scriptr_1(x)$ becomes
\begin{equation} \label{0.5again}
|x_d+|x'|^2-2x'\cdot s|\le |s|\le|x_d+|x'|^2|
\text{ for all } s\in\scriptr_1(x).
\end{equation}

If $|x'|\le 1$, then since $\max\{|x'|,|x_d+|x'|^2|\}\gg 1 $, necessarily $|x_d|\gg 1$.
Since $|x_d+|x'|^2-2x'\cdot s|\le|s|$,
it follows that $|x_d|\le C |s|$.
On the other hand, by \eqref{0.5again} again,
$|s|\le |x_d+|x'|^2|\le 2|x_d|$. So $|s|\asymp |x_d|$.
Therefore
$$J_1(x)
\lesssim  \int_{|s|\le 2|x_d|} |x_d|^{-d}ds\lesssim {|x_d|^{-1}}\sim |x_d+|x'|^2|^{-1}
\sim \upsilon_*(x)^{1/d}.$$

Suppose now that $|x'|\ge 1$.
Recall our standing assumption that
$|x_1|\ge |x'|/\sqrt{d-1}$.
Since
$|x_d+|x'|^2-2x'\cdot s|\le |s|$,
\[|x_d+|x'|^2|\le |s|+2|x'|\cdot|s|\le 3|x'|\cdot|s|.\]

Suppose that there exists some $s\in\scriptr_1(x)$ satisfying $|s|\le 1$,
and consider the contribution to $J_1(x)$ made by all such $s$.
For fixed $\tilde s$, according to \eqref{0.5again},
$s_1$ lies in an interval
of length $\lesssim |x'|^{-1}|s|\le|x'|^{-1}$.
Since $|\tilde s|\le |s|\le 1$,
the intersection of $\scriptr_1(x)$ with $\{s: |s_1|\le 1\}$
has measure $\lesssim |x'|^{-1}$.
Therefore
the contribution made to $J_1(x)$ by all $s\in\scriptr_1(x)$ satisfying
$|s|\le 1$ is
$\lesssim |x'|^{-1}$.
This is the required bound,
for since
$|x_d+|x'|^2|\le 3|x'|\cdot|s|$  and $|s|\le 1$,
\[
|x'|^{-1}
\lesssim |x_d+|x'|^2|^{-1}.
\]

Continuing to assume that $|x'|\ge 1$,
consider next the contribution of all $s\in\scriptr_1(x)$ satisfying $|s|>1$.
Now $s\in\scriptr_1$ implies that $|x_d+|x'|^2|\le |s|+2|s||x'|\le 3|s|\cdot|x'|$,
that is,
\begin{equation} \label{sbound1}
|s|\ge 3^{-1}|x'|^{-1}|x_d+|x'|^2|.
\end{equation}
We will consider two subcases, (i) $|x_d+|x'|^2-2x'\cdot s|> 3^{-1}|x'|^{-1}|x_d+|x'|^2|$
and (ii) $|x_d+|x'|^2-2x'\cdot s|\le  3^{-1}|x'|^{-1}|x_d+|x'|^2|.$

First, to treat the contribution of those $s\in\scriptr_1(x)$
which satisfy
$|x_d+|x'|^2-2x'\cdot s|> 3^{-1}|x'|^{-1}|x_d+|x'|^2|$,
define $E_2^{j,k}(x)$ to be the set of all $s\in \scriptr_1(x)$ which satisfy
both of
\begin{gather*}
 2^{k}3^{-1}|x'|^{-1}|x_d+|x'|^2|<|x_d+|x'|^2-2x'\cdot s|\le 2^{k+1}3^{-1}|x'|^{-1}|x_d+|x'|^2|
\\
2^{j}3^{-1}|x'|^{-1}|x_d+|x'|^2|\le |s|<2^{j+1}3^{-1}|x'|^{-1}|x_d+|x'|^2|.
\end{gather*}
If $s\in  E_2^{j,k}(x)$ then $k\le j+1$.

Now
\[|E_2^{j,k}(x)|\lesssim
(2^{j+1}|x'|^{-1}|x_d+|x'|^2)^{d-2}
2^{k+1}|x_d+|x'|^2|\cdot |x'|^{-2}.  \]
Indeed,
$|\tilde s|\le |s|\le 2^{j+1}3^{-1}|x'|^{-1}|x_d+|x'|^2|$,
while
for fixed $\tilde s$, $s_1$ lies in an interval
of length $\sim 2^{k+1}|x_d+|x'|^2|\cdot |x'|^{-2}$.

Together with \eqref{0.5again}, this implies that
the total contribution made to $J_1(x)$
by all $s\in\scriptr_1(x)$ which satisfy
$|x_d+|x'|^2-2x'\cdot s|> 3^{-1}|x'|^{-1}|x_d+|x'|^2|$
is
\begin{equation*}\aligned
&\le \sum_{j=1}^{[\log_2|x_1|+1] }\sum_{k=1}^{j+1}
\int_{E_2^{j,k}(x)}  {\abr{ {s}}^{-d}} \, ds
\\
&\lesssim\sum_{j=1}^\infty \sum_{k=1}^{j+1}
(2^{j}|x'|^{-1}|x_d+|x'|^2|)^{-d}|E_2^{j,k}(x)|
\\
&\lesssim
\sum_{j=1}^\infty \sum_{k=1}^{j+1}(2^{j}|x_d+|x'|^2|\cdot|x'|^{-1})^{-d}
(2^{{j+1}}|x'|^{-1}|x_d+|x'|^2|)^{d-2}2^{k+1}|x'|^{-2}|x_d+|x'|^2|
\\
&\lesssim
(|x_d+|x'|^2|\cdot|x'|^{-1})^{-d}
\cdot
(|x'|^{-1}|x_d+|x'|^2|)^{d-2}
\cdot |x'|^{-2}|x_d+|x'|^2|
\cdot
\sum_{j=1}^\infty 2^{-j}
\\
&=|x_d+|x'|^2|^{-1}.
\endaligned
\end{equation*}

Secondly, to treat the contribution of those $s\in\scriptr_1(x)$ which satisfy
$|x_d+|x'|^2-2x'\cdot s|\le 3^{-1} |x'|^{-1}|x_d+|x'|^2|$,
define $E_3^{j}(x)$ to be the set of all
$s\in \scriptr_1(x)$
which satisfy both of
\begin{gather}
 |x_d+|x'|^2-2x'\cdot s|\le 3^{-1} |x'|^{-1}|x_d+|x'|^2|
\\
2^{j}|x'|^{-1}|x_d+|x'|^2|<|s|\le 2^{j+1}|x'|^{-1}|x_d+|x'|^2|.
\end{gather}
If $s\in E_3^{j}(x)$ and $\tilde s$ is fixed, then $s_1$ lies in
an interval of length
$\lesssim |x'|^{-2}|x_d+|x'|^2|$, while $|\tilde s|\le |s|\le 2^{j+1}|x'|^{-1}|x_d+|x'|^2|$.
Thus
\[
|E_3^j(x)|\lesssim
|x'|^{-2}|x_d+|x'|^2| \cdot (2^{j}|x'|^{-1}|x_d+|x'|^2|)^{d-2}
=2^{(d-2)j}|x'|^{-d}|x_d+|x'|^2|^{d-1}.
\]
Consequently the contribution made by all $s\in\scriptr_1(x)$
which satisfy
$|x_d+|x'|^2-2x'\cdot s|\le 3^{-1} |x'|^{-1}|x_d+|x'|^2|$
is
\begin{multline*}
\lesssim \sum_{j=1}^{[\log_2|x_1|+1] } \int_{E_3^{j}(x)}  {\abr{ {s}}^{-d}}  \,ds
\\
 \lesssim
\sum_{j=1}^\infty (2^{j}|x'|^{-1}|x_d+|x'|^2|)^{-d}
2^{(d-2)j}|x'|^{-d}|x_d+|x'|^2|^{d-1}
 \le |x_d+|x'|^2|^{-1}.
\end{multline*}
This completes the discussion of $J_1(x)$.

\bigskip
We turn to the discussion of
$J_2(x)=\int_{\scriptr_2(x)} \abr{x_d+|x'|^2-2x'\cdot s}^{-d}\,ds$.
As already noted, we may continue to assume that
$\max(|x'|,|x_d+|x'|^2|)\gg 1$.

\noindent
\bull
\textbf {Estimate for $J_2(x)$ in the case $|x'|\le|x_d+|x'|^2|$  and $|x'|>1$.}
If
$|x_d+|x'|^2-2x'\cdot s|\le \tfrac {|x_d+|x'|^2|}{4|x'|}$, then
$|2x'\cdot s|\ge |x_d+|x'|^2|(1-\tfrac 1{4|x'|})$. Therefore
\begin{equation} \label{eq:strangecontradiction}
|s|\ge\tfrac{ |x_d+|x'|^2|}{2|x'|}(1-\tfrac 1{4|x'|})
\ge \tfrac 38\frac{ |x_d+|x'|^2|}{|x'|}
\ge \tfrac32 |x_d+|x'|^2-2x'\cdot s|.
\end{equation}
This contradicts the
definition of $\scriptr_2(x)$.
We conclude that
if $|x'|\le|x_d+|x'|^2|$  and $|x'|>1$,
then
\[|x_d+|x'|^2-2x'\cdot s|> \tfrac {|x_d+|x'|^2|}{4|x'|}.\]

Define
\[
E_4^{j}(x)=\{s\in\scriptr_2(x):
2^{j}|x'|^{-1}|x_d+|x'|^2|<|x_d+|x'|^2-2x'\cdot s|
\le 2^{j+1}|x'|^{-1}|x_d+|x'|^2|.\}
\]
If
$s\in E_4^{j}(x)$ and $\tilde s$ is fixed, then
since $|x_1|\sim|x'|$,
$s_1$ lies in an interval
of length
$\sim \tfrac { 2^{j}|x_d+|x'|^2|}{|x'|^2}$.
From the bound
\[|\tilde s|\le|s|\le |x_d+|x'|^2-2x'\cdot s|
\le 2^{j+1}|x_d+|x'|^2||x'|^{-1},\]
it now follows that
\[
|E_4^j(x)|\lesssim
|x'|^{-2}2^{j}
|x_d+|x'|^2|
(2^{j}|x'|^{-1}|x_d+|x'|^2|)^{d-2}
=
2^{j-1}|x'|^{-d} |x_d+|x'|^2|^{d-1}.
\]
Therefore
if $|x'|\le|x_d+|x'|^2|$  and $|x'|>1$,
the contribution of $\scriptr_2(x)$ to $J_2(x)$ is
\begin{multline*}
\sum_{{j=0}}^\infty \int_{E_4^{j}(x)} \abr{x_d+|x'|^2-2x'\cdot s}^{-d}\,ds
\\
\lesssim \sum_{{j=0}}^\infty
(2^j\tfrac {|x_d+|x'|^2|}{|x'|})^{-d}
2^{j(d-1)}|x'|^{-d} |x_d+|x'|^2|^{d-1}
=|x_d+|x'|^2|^{-1}.
\end{multline*}

\noindent
\bull
\textbf {Estimate for $J_2(x)$ in the case $|x'|\le|x_d+|x'|^2|$  and $|x'|\le 1$.}
Let $s\in\scriptr_2(x)$.
Then
\begin{align*}
|x_d+|x'|^2-2x'\cdot s|
&\ge |x_d+|x'|^2|-2|s\cdot x'|
\\
&\ge |x_d+|x'|^2|-2|s|
\\
&\ge |x_d+|x'|^2|-2|x_d+|x'|^2-2x'\cdot s|,
\end{align*}
where the definition of $\scriptr_2(x)$ was invoked to obtain the
last inequality.
This implies that
\[|x_d+|x'|^2-2x'\cdot s|\ge \tfrac13 |x_d+|x'|^2|. \]

Define
\[
E_5^{j}(x)=\{s\in\scriptr_2(x):
2^j|x_d+|x'|^2|
\le 3|x_d+|x'|^2-2x'\cdot s|
<2^{j+1} |x_d+|x'|^2|.\}
\]
Any $s\in E_5^j(x)$ satisfies
\[
|s|\le
|x_d+|x'|^2|-2|x_d+|x'|^2-2x'\cdot s|\sim 2^j|x_d+|x'|^2|,\]
so
$|E_5^j(x)|\lesssim (2^j|x_d+|x'|^2|)^{d-1}$,
whence
\begin{multline*}
J_2(x)
=\int_{\scriptr_2(x)}
\abr{x_d+|x'|^2-2x'\cdot s}^{-d}\,ds
\lesssim
\sum_{j=0}^\infty
(2^j|x_d+|x'|^2|)^{-d}
|E_5^j(x)|
\\
\lesssim
\sum_{j=0}^\infty
(2^j|x_d+|x'|^2|)^{-d}
(2^j|x_d+|x'|^2|)^{d-1}
= C |x_d+|x'|^2|^{-1}.
\end{multline*}
This concludes the analysis of $J_2(x)$, in the case in which $|x'|\le|x_d+|x'|^2|$.

\noindent
\bull
\textbf {Estimate for $J_2(x)$ in the case $|x'|\ge|x_d+|x'|^2|$.}
We may continue to assume that
$|x_1|\ge |x'|/\sqrt{d-1}$.
Partition $\scriptr_2(x)$ into the following three subregions:
\begin{gather*}
|x_d+|x'|^2-2x'\cdot s|>2|x'|,
\\
1<|x_d+|x'|^2-2x'\cdot s|\le 2|x'|
\\
|x_d+|x'|^2-2x'\cdot s|\le 1.
\end{gather*}


To analyze the contribution of the
subregion in which $|x_d+|x'|^2-2x'\cdot s|>2|x'|$,
for each integer $j\ge 1$ define
\[
E_7^j(x)
=\{
s\in\scriptr_2(x):
2^j{|x'|}<|x_d+|x'|^2-2x'\cdot s|\le 2^{j+1}{|x'|}
\}.
\]
For fixed $\tilde s$,
$s_1$ lies in a subinterval of length $\sim 2^j$,
while
$|\tilde s|
\le |s|\le |x_d+|x'|^2-2x'\cdot s|\lesssim 2^j|x'|,
$
so
$|E_7^j(x)|\lesssim 2^{j(d-1)}|x'|^{d-2}$.
Therefore
the contribution of this subset of $\scriptr_2(x)$ to $J_2(x)$ is
\begin{multline*}
\lesssim \sum_{j=1}^\infty (2^j|x'|)^{-d}|E_7^j(x)|
\lesssim \sum_{j=1}^\infty (2^j|x'|)^{-d}2^{j(d-1)}|x'|^{d-2}
=|x'|^{-2}
\ll |x'|^{-1}.
\end{multline*}

To analyze the contribution of the subregion in which
$1<|x_d+|x'|^2-2x'\cdot s|\le 2|x'|$, partition further into
subregions in which
$|x_d+|x'|^2-2x'\cdot s|\le 2|x'|\sim 2^{-j}|x'|$,
where $1\le 2^j\lesssim |x'|$.
Such a sub-subregion has measure $\lesssim 2^{-j}(2^{-j}|x'|)^{d-2}
= 2^{-j(d-1)}|x'|^{d-2}$,
giving a total contribution to $J_2(x)$ which is
\[
\lesssim \sum_{1\le 2^j\lesssim |x'|}
(2^{-j}|x'|)^{-d} 2^{-j(d-1)}|x'|^{d-2}
= \sum_{1\le 2^j\lesssim |x'|}
2^j |x'|^{-2}
\lesssim |x'|^{-1},
\]
as required.

It remains only
to analyze the contribution made to $J_2(x)$ by the
subregion in which $|x_d+|x'|^2-2x'\cdot s|\le 1$,
assuming still that $|x'|\ge |x_d+|x'|^2|$.
Since $|s|\le |x_d+|x'|^2-2x'\cdot s|\le 1$,
and since $s_1$ lies in an interval of length $\lesssim |x_1|^{-1}\lesssim|x'|^{-1}$
so long as $\tilde s$ remains fixed,
the measure of this subregion is $\lesssim |x'|^{-1}$.
The integrand $\abr{x_d+|x'|^2-2x'\cdot s}^{-d}$ is $\le 1$,
so its integral over this subregion is $\lesssim |x'|^{-1}$.

The proof of Lemma~\ref{lemma:key} is complete.
\end{proof}

\begin{corollary} Let $f_0=\chi_{|x|\le 1}$.
Define $f_{n+1}$ to be
$(T(f_n))^d$ if $n$ is even,
and
$(T^*(f_n))^d$ if $n$ is odd.
Then
for any even $n>0$,
\[f_n\le C^n \upsilon_*.\]
\end{corollary}
This follows at once from $n$ applications of Lemma~\ref{lemma:key},
since $0\le f_0\le\upsilon_*$.

\section{Proof of Lemma~\ref{lemma:leibniz} } \label{section:leibniz}

The following argument is essentially taken from \cite{christweinstein}.

\begin{proof}
Fix a smooth, compactly supported
cutoff function $\eta\in C^\infty_0(\reals^d)$
satisfying $\eta(\xi)\equiv 1$ for all $|\xi|\le 1$,
and $\eta(\xi)=0$ for all $|\xi|\ge 2$.
For each $k\in\{0,1,2,\cdots\}$
introduce the Fourier multiplier $P_{k}$ defined by
$\widehat{P_k f}(\xi) = \widehat{f}(\xi)\eta(2^{-k}\xi)$.
For $k\ge 1$ define $Q_k=P_k-P_{k-1}$.
Observe  that $\widehat{Q_k f}(\xi)$
is supported in $\{\xi: 2^{k-1}\le |\xi|\le 2^{k+1}\}$.
Fix $K$ such that $2^K\ge\Lambda>2^{K-1}$.
Set $R_K f = f-P_{K}f$.
$\widehat{R_K f}(\xi)$ is supported in
$\{\xi: |\xi|\ge 2^{K}\}$.

Decompose
\[f
= P_0 f+\sum_{k=1}^{K} Q_k f + R_K f,
\]
and decompose $g$ in the same manner.
By expanding the product $fg$ in terms of these summands and
recombining terms, one obtains
\begin{align}
fg &= \sum_{k=3}^{K} Q_k f\cdot P_{k-3}g
+ \sum_{k=3}^K Q_k g\cdot P_{k-3}f
\label{toomanyterms1}
\\
&+ \sum_{k=2}^K Q_k f\big (Q_{k-2}g+Q_{k-1}g+Q_k g\big)
+ \sum_{k=2}^K Q_k g\big (Q_{k-2}f+Q_{k-1}f\big)
\label{toomanyterms2}
\\
&+ R_K f\cdot P_{K-2}g
+ R_K g\cdot P_{K-2}f
\label{toomanyterms3}
\\
&+ R_K f \big(Q_{K-1}g + Q_K g\big)
+ R_K g \big(Q_{K-1}f + Q_K f\big)
+ R_K f\cdot R_K g
\notag
\\
&+\scriptr(f,g)
\notag
\end{align}
where $\scriptr(f,g)$
is a constant-coefficient finite linear combination of twofold products
of the factors $P_0f,P_0g, Q_1f, Q_1g$.

Consider the contribution made to $D_\Lambda^s(fg)$ by the first term
on the right-hand side in this equation.
The Fourier transform of
$D_\Lambda^s\big(Q_k f\cdot P_{k-3}g\big)$
is supported in
$\{\xi:  2^{k-2}\le |\xi|\le 2^{k+2}\}$.
Therefore by weighted Littlewood-Paley theory \cite{stein},
since $u\in L^r(\reals^d)$,
\[
\norm{
\sum_{k=2}^K D_\Lambda^s\big(Q_k f\cdot P_{k-3}g\big)
}_{L^r(u)}
\asymp
\norm{
\big\{D_\Lambda^s\big(Q_k f\cdot P_{k-3}g\big)\big\}_{k=2}^K
}_{L^r(\ell^2)(u)}
\]
where
\[
\norm{\{h_k\}}_{L^r(\ell^2)(u)}^r
= \int_{\reals^d} \big(\sum_k |h_k(x)|^2\big)^{r/2}u(x)\,dx.
\]

Writing $\widehat{D_\Lambda^s f}(\xi) = \lambda_{s,\Lambda}(\xi)\widehat{f}(\xi)$,
let $m_k(\xi) = 2^{-ks} \lambda_{s,\Lambda}(\xi)\zeta(2^{-k}\xi)$
where $\zeta(\xi)\equiv 1$ whenever $\tfrac14\le|\xi|\le 4$,
and
$\zeta(\xi)\equiv 0$ whenever $|\xi|\le \tfrac18$
or $|\xi|\ge 8$.

Define
$\widehat{M_k f}(\xi) = m_k(\xi) \widehat{f}(\xi)$.
Then
\[
\norm{
\big\{D_\Lambda^s\big(Q_k f\cdot P_{k-3}g\big)\big\}_{k=2}^K
}_{L^r(\ell^2)(u)}
=
\norm{
\big\{ M_k \big(2^{ks}Q_k f\cdot P_{k-3}g\big)\big\}_{k=2}^K
}_{L^r(\ell^2)(u)}.
\]
Because $u\in A_r$ and the operator $\vec{M}\{h_k\}=\{M_k h_k\}$
is a vector-valued Calder\'on-Zygmund operator,
$\vec{M}$ is bounded on $L^r(\ell^2)(u)$ \cite{AJ}.
Thus
\[
\norm{
D_\Lambda^s
\sum_{k=2}^K
Q_k f\cdot P_{k-3}g
}_{L^r(u)}
\lesssim
\norm{
\big\{ 2^{ks}Q_k f\cdot P_{k-3}g\big\}_{k=2}^K
}_{L^r(\ell^2)(u)}.
\]

Now \[|2^{ks}Q_k f\cdot P_{k-3}g|\le 2^{ks}|Q_k f|\cdot \scriptm g\]
where $\scriptm$ denotes the Hardy-Littlewood maximal function.
Therefore by H\"older's inequality and the factorization $u=u_1v_1$,
\begin{align*}
\norm{
\big\{ 2^{ks}Q_k f\cdot P_{k-3}g\big\}_{k=2}^K
}_{L^r(\ell^2)(u)}
&\le C
\norm{
\scriptm g\cdot (\sum_{k=2}^K |2^{ks}Q_k f|^2)^{1/2}
}_{L^r(u)}
\\
&\le C
\norm{ \scriptm g}_{L^{q_1}(v_1^{q_1/r})}
\norm{
\big\{ 2^{ks}Q_k f\big\}_{k=2}^K
}_{L^{p_1}(\ell^2)(u_1^{p_1/r})}.
\end{align*}
Since $v_1^{q_1/r}\in A_{q_1}$
and $\scriptm$ is bounded on $L^{q_1}$
with respect to any weight in $A_{q_1}$ \cite{stein}, this is majorized by
\[
\norm{ g}_{L^{q_1}(v_1^{q_1/r})}
\norm{
\big\{ 2^{ks}Q_k f\big\}_{k=2}^K
}_{L^{p_1}(\ell^2)(u_1^{p_1/r})}.
\]
Again by weighted vector-valued Calder\'on-Zygmund theory \cite{AJ},
since $u_1^{p_1/r}\in A_{p_1}$,
the second factor in this expression is majorized by
$C\norm{D_\Lambda^s f}_{L^{p_1}(u_1^{p_1/r)}}$.
Therefore when $D_\Lambda^s$ is applied to the first term
on the right-hand side of \eqref{toomanyterms1},
a bound of the required form is obtained.

The contributions of the second term on the right
in \eqref{toomanyterms1},
and of both terms in \eqref{toomanyterms2},
are treated in the same way.
To treat the contribution of $\scriptr(f,g)$ requires only
H\"older's inequality, since only low values of $|\xi|$
come into play and $s\ge 0$.

We discuss next the contribution of $\sum_{k=2}^K Q_k f\cdot Q_k g$.
The summand
$Q_k f\cdot Q_k g$ has Fourier transform supported in $\{\xi: |\xi|\le 2^{k+2}\}$
and therefore
\[
D_\Lambda^s\big( \sum_{k=2}^K Q_k f\cdot Q_k g\big)
=
\sum_{k=2}^K M_k\big(2^{ks}Q_k f\cdot Q_k g \big)
\]
where $M_k$ is the Fourier multiplier operator with multiplier
\[
m_k(\xi) = \lambda_{s,\Lambda}(\xi) 2^{-ks} \eta(2^{-k-2}\xi).
\]
It is routine to verify, using the hypothesis that $s\ge 0$,
that
$|M_k h|\le C\scriptm(h)$ for any function $h$, uniformly
in $k,\Lambda$ for $0\le k\le K$.
Therefore
\begin{align*}
\norm{
D_\Lambda^s\big( \sum_{k=2}^K Q_k f\cdot Q_k g\big)
}_{L^r(u)}
&\le C
\norm{
\scriptm
\sum_{k=2}^K Q_k f\cdot Q_k g
}_{L^r(u)}
\\
&\le C
\norm{
\sum_{k=2}^K 2^{ks}Q_k f\cdot Q_k g
}_{L^r(u)}
\\
&\le C
\norm{
\big(\sum_{k=2}^K |2^{ks}Q_k f|^2\big)^{1/2}
\big(\sum_{k=2}^K |Q_k g|^2\big)^{1/2}
}_{L^r(u)}
\\
&\le C
\norm{ \{2^{ks}Q_k f\}_{k=2}^K
}_{L^{p_1}(\ell^2)(u_1^{p_1/r})}
\cdot
\norm{ \{Q_k g\}_{k=2}^K
}_{L^{q_1}(\ell^2)(v_1^{q_1/r})}
\\
&\le C
\norm{D_\Lambda^s f
}_{L^{p_1}(u_1^{p_1/r})}
\norm{g
}_{L^{q_1}(\ell^2)(v_1^{q_1/r})},
\end{align*}
as desired.

All remaining terms can be treated in the same way as we have done for
$\sum_{k=2}^K Q_k f\cdot Q_k g$.
\end{proof}

\end{document}